\documentclass[11pt, reqno,sumlimits]{amsart}
\usepackage{amssymb,amscd, epsfig, mathrsfs, xypic, amsmath,amsthm, color} 
\usepackage{hyperref}
\usepackage{blindtext}
\usepackage{enumerate}

\xyoption{all}
\newtheorem{thm}{Theorem}[section]  
\newtheorem{cor}[thm]{Corollary}
\newtheorem{lem}[thm]{Lemma}
\newtheorem{prop}[thm]{Proposition} 
\newtheorem{df-pr}[thm]{Definition-Proposition}

\theoremstyle{definition}
\newtheorem{defn}[thm]{Definition} 
 
\newtheorem{rem}[thm]{Remark}

\newtheorem{exm}[thm]{Example}

\newcommand{\NN}{{\mathbb N}}

\newcommand{\CC}{{\mathbb C}}

\newcommand{\ZZ}{{\mathbb Z}}

\newcommand{\PP}{{\mathbb P}}
\newcommand{\LL}{{\mathbb L}}

\newcommand{\sfn }{{\mathsf n}}

\newcommand{\bfk}{{\mathbf k }}

\newcommand{\bfx}{{\mathbf x }}
\newcommand{\bfm}{{\mathbf m }}

\newcommand{\bfs}{{\mathbf s}}

\newcommand{\calL}{{\mathcal L}}

\newcommand{\calP}{{\mathcal P}}
\newcommand{\calQ}{{\mathcal Q}}
\newcommand{\calR}{{\mathcal R}}

\newcommand{\scP}{{\mathscr P}}

\newcommand{\scS}{{\mathscr S}}

\newcommand{\surj}{\twoheadrightarrow}

\newcommand{\Fl}{\operatorname{Fl}}

\newcommand{\Ker}{\operatorname{Ker}}

\newcommand{\Gr}{\operatorname{Gr}}

\newcommand{\gr}{\operatorname{gr}}

\newcommand{\Spec}{\operatorname{Spec}}

\newcommand{\rk}{{\operatorname{rk}}}

\newcommand{\id}{{\operatorname{id}}}

\newcommand{\supp}{{\operatorname{supp}}}

\newcommand{\SM}{\mathbf{Sm}_{\bfk}}

\newcommand{\CK}{{\it CK}}
\newcommand{\CH}{{\it CH}}

\newcommand{\srarrow}{\twoheadrightarrow}
\newcommand{\irarrow}{\hookrightarrow}

\newcommand{\Laz}{\mathbb{L}}
\newcommand{\trecd}{\cdot\cdot\cdot}
\newcommand{\tred}{\ldots}

\newsavebox{\savepar}

\numberwithin{equation}{section}

\newcounter{labelflag} 
\newcommand{\labelon}{\setcounter{labelflag}{1}}
\newcommand{\Label}[1]{\ifnum\thelabelflag=1\ifmmode
\makebox[0in][l]{\qquad\fbox{\rm#1}} \else
\marginpar{\vspace{0.7\baselineskip} \hspace{-1.1\textwidth}
\fbox{\rm#1}} \fi \fi \label{#1} } \labelon

\begin{document} 
\title{Segre classes and Damon--Kempf--Laksov formula in algebraic cobordism}
\author{Thomas Hudson, Tomoo Matsumura}
\date{}
\maketitle 
\begin{abstract}
In this paper, we introduce (relative) Segre classes for algebraic cobordism and prove a formula for their generating function. As an application, we prove a generalisation of the determinantal formula for the fundamental class of degeneracy loci to the algebraic cobordism of Grassmann bundles.
\end{abstract}

\section{Introduction}

 Given a sufficiently general morphism of vector bundles $\varphi:V\rightarrow W$ over a smooth quasi-projective variety $X$, the Giambelli--Thom--Porteous formula describes the fundamental class of the degeneracy locus $D_r(V,W)$, the variety consisting of the points over which $\text\,\varphi$ has dimension at most $r$, as a determinant in the Chern classes of the given bundles. This formula was then extended by Damon (\cite{Damon1973}), Kempf--Laksov  (\cite{KempfLaksov}) and, later, by Fulton (\cite{FlagsFulton}), who considered the following more general setting:
one replaces $V$ and $W$ with flags $V_{p_1}\subseteq\cdots \subseteq V_{p_d}=F$ and $W=W_{q_d}\srarrow\cdots \srarrow W_{q_1}$ (here the subscripts represent the ranks of the bundles) and the degeneracy locus $D_{r_\bullet}$ is then given by the intersection of the loci $D_{r_i}(V_{p_i},W_{q_i})$, where the $d$-tuple $r_\bullet=(r_1,\cdots,r_d)$ is required to satisfy some conditions. Fulton was able to express the fundamental class $[D_{r_\bullet}]$ as a Schur determinant in the Chern classes of the bundles $V_{p_{i+1}}/V_{p_i}$ and $\Ker(W_{q_{i+1}}\srarrow W_{q_i})$. By considering the case in which one has $p_i-r_i=i$ and all the $W_{q_i}$'s coincide, one recovers the Damon--Kempf--Laksov formula.

In exactly the same way in which Giambelli's original result can be translated into a description through Schur functions of the Schubert classes of the Grassmannian of $d$-dimensional planes $\Gr_d(\CC^e)$ via a closed, determinantal formula in the Chern classes of the tautological bundle $U$, the Damon--Kepf--Laksov formula can be reinterpreted as a description of the fundamental classes of the Schubert varieties of Grassmann bundles.
 More precisely, we consider a vector bundle $E\rightarrow X$ of rank $e$ and we fix a reference flag of subbundles $0=F^e\subset \trecd\subset F^1\subset F^0=E$, setting $F_{\ell}:=E/F^{\ell}$. Then to a partition $\lambda=(\lambda_1,\tred,\lambda_r)$ with $r\leq d$ and $\lambda_1\leq e-d$ we associate, inside the Grassmann bundle $\Gr_d(E)$, the Schubert variety $X_\lambda$ obtained by selecting the pairs $(x,U_x)\in \Gr_d(E)$ for which $\text{dim}\,(F^{\lambda_i-i+d}\cap U_x)\geq i$ for all $i$. As an element of $\CH^*(\Gr_d(E))$ its fundamental class is given by the Schur determinant
\begin{align}\label{eqn Schur}
[X_\lambda]_{\CH} = \Delta_{\lambda}\Big(c(1),\dots, c(r)\Big) := \det\Big(c(i)_{\lambda_i +j-i}\Big)_{1\leq i,j\leq r},
\end{align}
where $c(i)$ denotes the total Chern class $c(F_{\lambda_i-i+d}-U)=c(F_{\lambda_i-i+d})/c(U)$.

The goal of this paper is to generalise this expression to algebraic cobordism and to develop the tools necessary for its proof. Algebraic cobordism, denoted $\Omega^*$, was introduced by Levine and Morel in \cite{LevineMorel} and represents the universal object among oriented cohomology theories, a family of functors which includes both the Chow ring $\CH^*$ and $K^0[\beta,\beta^{-1}]$, a graded version of the Grothendieck ring of vector bundles.
By definition, an oriented cohomology theory $A^*$ is equipped with pullbacks, pushforwards for projective morphisms and a theory of Chern classes for vector bundles $c^A$. Hence one might hope to be able to extend (1.1), provided one can understand the role played in it by the formal group law $F_A(u,v)$ and its formal inverse $\chi_A(u)$. These are power series over $A^*(\Spec (\bfk))$ which respectively describe the behaviour of $c_1^A$ on line bundles with respect to tensor product and taking the dual.  


A careful inspection of the original proofs, in which $[X_\lambda]_{CH}$ is computed by pushing forward to $Gr_d(E)$ the fundamental class of a resolution of singularities of $X_\lambda$ denoted 
$$\psi:Y_\lambda\rightarrow X_\lambda\irarrow Gr_d(E),$$
 convinced us that it is more natural to express (\ref{eqn Schur}) in terms of the Segre classes $s(\text{-})$. Given that for any bundle $V$ one has $c(V)=s(-V^\vee)$, this alternative formulation reads
\begin{align*}
[X_\lambda]_{\CH} =\psi_*[Y_\lambda]_{CH}= \Delta_{\lambda}\Big(s(1),\dots, s(r)\Big)
\end{align*}      
where $s(i)$ represents the total Segre class $s\big((U-F_{\lambda_i-i+d})^\vee\big)$.

In this format the formula does generalise, provided that one introduces an appropriate notion of Segre classes, denoted $\scS^A$, and a power series $P_A(z,x)$, defined as the unique solution to the equation $F_A(z,\chi_A(x))=(z-x)P_A(z,x)$. More specifically we obtain the following result.     

\vspace{0.4 cm}

\noindent{\bf Theorem A }
({\it cf.} Theorem \ref{detthm}).
For any partition $\lambda \in \calP_d(n)$ of length $r$,  we have
\begin{align}\label{eqn THMA}
[Y_\lambda \to \Gr_d(E)]_A :=\psi_*[Y_\lambda]_A= \sum_{\bfs \in\NN^r}a_{\bfs} \Delta_{\lambda+\bfs}\Big(\scS^A(1), \dots, \scS^A(r)\Big), 
\end{align}
where $\scS^A(i)=\scS^A\big((U-F_{\lambda_i-i+d})^\vee\big)$ and the coefficients $a_{\bfs}\in A^*(\Spec (\bfk))$ are given by
\[
\prod_{1\leq i<j\leq r}P_A(t_j,t_i)=\sum_{\bfs=(s_1,\dots, s_r) \in\NN^r}a_{\bfs}\cdot t_1^{s_1}\cdots t_r^{s_r}.
\]

\vspace{0.2 cm}

In the special case $A^*=CH^*$ the two notions of Segre classes actually coincide and moreover $P_{CH}(z,x)=1$, therefore our expression recovers the classical statement since the only non zero coefficient is $a_{(0,\tred,0)}$. For a less trivial application involving formal group laws given by polynomials, we refer the reader to \cite{HudsonMatsumuraInf} in which we described more explicitly the case of infinitesimal cohomology theories.

 
 In order to better appreciate the significance of our formula, it may be worth placing it within the wider framework of generalised Schubert calculus. In recent years a lot of effort has been devoted to lift results of classical Schubert calculus to $\Omega^*$, in a fashion similar to what Bressler--Evens did in \cite{BraidBressler,SchubertBressler} for topological cobordism. In particular, this line of research was pioneered by Calm\'{e}s--Petrov--Zanoulline (\cite{SchubertCalmes}) and Hornbostel--Kiritchenko (\cite{SchubertHornbostel}) who studied the algebraic cobordism of flag manifolds. 
  Later, the attention shifted to flag bundles with contributions given by Kiritchenko--Krishna (\cite{EquivariantKiritchenko}), Calmes--Zainoulline--Zhong (\cite{EquivariantCalmes}) and the first author (\cite{ThomHudson,GeneralisedHudson}).

It should be noticed, however, that the interpretation of such results requires a little caution. On the one hand this is due to the fact that not all Schubert varieties have a well defined notion of fundamental class, only those that are local complete intersection schemes. On the other hand, all classical techniques inherently depend on the choice of a resolution, which in general is not unique. This dependence can be also observed when one considers the  polynomials describing the classes: so far for $\Omega^*$ it has not been possible to identify stable representatives, \textit{i.e.} which are independent of $e$. From this perspective $[Y_\lambda \to \Gr_d(E)]_\Omega$, which we refer to as the \emph{Damon--Kempf--Laksov class}, has the advantage of being stable along the natural maps of the infinite system of Grassmann bundles obtained by increasing the size of the vector bundle $E$. As a consequence it becomes possible to define a generalisation of Schur/Grothendieck polynomials in the context of algebraic cobordism, an aspect of the theory that we intend to develop in our future works.




Within this framework, the decision to restrict our attention to the Schubert varieties $X_\lambda$ of the Grassmann bundles is justified by the need to have at our disposal the associated Damon--Kempf--Laksov resolution $Y_\lambda$ (see Definition \ref{defKL}). Such an approach was also at the core of  \cite{HIMN} (joint with T. Ikeda and H. Naruse) in which we managed to establish a new determinantal formula for the Schubert classes $[X_\lambda]_{\CK}$ of connective $K$-theory, an oriented cohomology theory obtained from $\Omega^*$ which can be specialised to both $\CH^*$ and $K^0$. In fact, since $X_{\lambda}$ has at worst rational singularities, we were able to conclude that $[Y_\lambda \to \Gr_d(E)]_{\CK}$ actually coincides with the fundamental class of $X_{\lambda}$ in $\Gr_d(E)$. (As an aside let us mention that, although (\ref{eqn THMA}) does not turn directly into the original expression we found, it still produces a single determinant, see Corollary~\ref{exmCK}.)
The key point of the proof given in \cite{HIMN} is to combine the geometric input given by the resolution with an algorithmic procedure modelled after the one used in \cite{Kazarian} by Kazarian to describe the Schubert classes of maximal isotropic Grassmann bundles of symplectic and orthogonal type.  For a detailed historical account and a more in depth description, we refer the interested reader to  \cite{HIMN},  let us simply mention that the key insight of Kazarian is that it is possible to reduce complex manipulations of Chern and Segre classes to computations with Laurent series.

In order to be able to make use of Kazarian's machinery in the more general setting, we had to introduce a notion of Segre classes for oriented cohomology theories and find a way to describe them as explicitly as possible. The definition we use is the exact analogue of the one given by Fulton in \cite{FultonIntersection}. For an oriented cohomology theory $A^*$ and a bundle $V\rightarrow X$ of rank $n$ we set $\scS^A_i(V):= \pi_*\Big(c_1^A\big(O(1)\big)^{i+n-1}\Big)$ where $\PP^*(V)\stackrel{\pi}\longrightarrow X$ is the associated dual projective bundle with tautological quotient bundle $O(1)$.  We then manage to relate the Segre polynomial $\scS^A(V;u)$ to the  Chern polynomial $c^A(V;u)$ by the formula
\begin{equation}\label{eqnSegre}
\scS^A(V;u)c^A(V;-u)=\frac{\scP^A(u)}{w^A(V;u)},
\end{equation}
where $\scP^A(u):=\sum_{i}[\PP^i]_A\cdot u^{-i}$ and $w^A(V;u)$ is a series constructed by making use of $P_A$ (see Definition \ref{defwtilde1}). It should be noticed that the right hand side is a power series in $u^{-1}$ and that (\ref{eqnSegre}) generalises the classical Chow ring identity $s(V;u)c(V;-u)=1$, which allows one to interpret Segre classes as the complete symmetric functions in the Chern roots. 

 An important consequence of 
our description is that it becomes possible to lift the definition of Segre classes to the Grothendieck ring of vector bundles, so that they can be evaluated on virtual bundles. The following result provides a geometric interpretation to such extension. 
\vspace{0.4 cm}

\noindent{\bf Theorem B }({\it cf.} Theorem \ref{relSegrePush}).
Let $V$ and $W$ be two vector bundles over $X$, respectively of rank $n$ and $m$. Consider the dual projective bundle $\PP^*(V)\stackrel{\pi}\longrightarrow X$ with tautological bundle $O(1)$. Then, for every oriented cohomology theory $A^*$ one has
\[
\scS^A_{m-n+1}(V-W)=\pi_*\Big(c_m^A\big(O(1)\otimes W^\vee\big)\Big)
\]
as elements of $A^*(X)$.

 After our work was completed, we were informed by Nakagawa--Naruse that in \cite{NakagawaNaruse} they achieved, by considering a different resolution, a stable generalisation of the Hall--Littlewood type formulas for Schur polynomials in the context of topological cobordism ({\it cf.} \cite{NakagawaNaruse1}). Finally, it is known from the work of Lascoux--Schutzenberger \cite{LascouxSchutzenberger} and Fulton \cite{FlagsFulton} that in cohomology one can express the degeneracy loci classes associated to \emph{vexillary} permutations as determinants ({\it cf.} Anderson--Fulton \cite{AndersonFulton2}). We expect that such result can be lifted to algebraic cobordism by using our method. 
\vspace{0.1 cm}
\textit{Notations and conventions:}
In this paper $\bfk$ stands for a field of characteristic 0 and $\SM$ is the category of smooth separated schemes of finite type which are quasi-projective over $\Spec (\bfk)$. Finally, we will follow the convention according to which 0 belongs to the natural numbers $\NN$. 

\section{Preliminaries on algebraic cobordism}
An oriented cohomology theory consists of a contravariant functor $A^*:\SM\rightarrow \calR^*$, together with a family of pushforward maps $\{f_*:A^*(X)\rightarrow A^*(Y)\}$, one for each projective morphism $f:X\rightarrow Y$. We will not describe in detail the compatibilities and the properties that this data is required to satisfy, the interested reader can find the precise definition  in \cite[Definition 1.1.2]{LevineMorel}. Instead, we will illustrate the aspects in which a general oriented cohomology theory differs from the Chow ring, the simplest and perhaps best known example, which the reader should always bear in mind as a first approximation. 

Since all oriented cohomology theories satisfy the projective bundle formula, each of them allows a theory of Chern classes which, in most respects, mirrors the one for $CH^*$: to every vector bundle $V\rightarrow X$ it is possible to associate a Chern polynomial $c^A(V;u)\in A^*(X)[u]$. Such assignment respects the Whitney formula, so that it can be extended to the Grothendieck group of vector bundles $K^0(X)$. To a class $[V]-[W]$ one associates 
\begin{align} \label{eq virtual}
c^A(V-W;u)=\frac{c^A(V;u)}{c^A(W;u)}\quad\text{ or, equivalently, }\quad c^A_i(V-W)=\sum_{j=0}^i(-1)^jc^A_{i-j}(V)  h^A_j(W),
\end{align}
 where $h^A_j(W)$ stands for the $j$-th complete symmetric function in the Chern roots of $W$.

 A close examination of the behaviour of the first Chern classes of line bundles unveils a key aspect in which $CH^*$ proves to be too limited to adequately represent all theories. While it is well known that $c_1^{CH}$ is linear with respect to tensor product, this is no longer true in general: describing $c_1^A(L\otimes M)$ in terms of the classes of the factors requires the use of a formal group law $(A^*(\Spec \bfk),F_A)$. This is a power series $F_A$ defined over the coefficient ring $A^*(\Spec (\bfk))$ such that, for any choice of line bundles $L$ and $M$ over some scheme $X$, one has 
$$c_1^A(L\otimes M)=F_A(c_1^A(L),c_1^A(M)).$$
 In a similar fashion, the usual equation $c_1^{CH}(L^\vee)=-c^{CH}_1(L)$ becomes $c_1^A(L^\vee)=\chi_A\big(c_1^A(L)\big)$, where $\chi_A\in A^*(\Spec (\bfk))[[u]]$  is the so-called formal inverse, the unique power series such that
$$F_A\big(u, \chi_A(u)\big)=0.$$

The main achievement of Levine and Morel concerning oriented cohomology theories is the construction of algebraic cobordism, denoted $\Omega^*$, which they identify as universal in the following sense.
\begin{thm}[(\protect{\cite[Theorems 1.2.6 and 1.2.7]{LevineMorel}})]
$\Omega^*$ is universal among oriented cohomology theories on $\SM$. That is, for any other oriented cohomology theory $A^*$ there exists a unique morphism 
$$\vartheta_A:\Omega^*\rightarrow A^*$$
of oriented cohomology theories. Furthermore, its associated formal group law $(\Omega^*(\Spec(\bfk)),F_\Omega)$ is isomorphic to the universal one defined on the Lazard ring $(\Laz,F)$.
\end{thm}
One of the consequences of the universality is that it allows to translate formulas which hold in $\Omega^*$ to every other oriented cohomology theory $A^*$, by making use of $\vartheta_A$. In particular, if the given formula has a classical version in either $\CH^*$ or $K^0$, then one is supposed to recover it. On the other hand, it is not always the case that properties that hold for the Chow ring or the Grothendieck ring will lift to algebraic cobordism.

 For example, one basic instance of this pheonomenon can be observed if one tries to compute the fundamental class of some closed subscheme $Z\stackrel{i_Z}\irarrow X$ in a smooth ambient space. While for the Chow ring it is sufficient to consider any resolution of singularities $\widetilde{Z}\stackrel{\varphi_{\widetilde{Z}}}\longrightarrow X$ to recover $[Z]_{\CH}$ as $\varphi_{\widetilde{Z}*}[\widetilde{Z}]_{\CH}$, for $K^0$ one is able to conclude that $[\mathcal{O}_Z]_{K^0}=\varphi_{\widetilde{Z}*}[\mathcal{O}_{\widetilde{Z}}]_{K^0}$ only if $Z$ has at worst rational singularities. Even this weaker statement proves to be false for algebraic cobordism, since different desingularisations can yield different push-forward classes. 

   On top of this lies an even bigger problem. As mentioned in the introduction, in $\Omega^*$ a scheme $Z\stackrel{\pi_Z}\longrightarrow \Spec(\bfk) $ has a well defined notion of fundamental class only if it is an l.c.i scheme. In fact, since l.c.i. pullbacks are available, one can make use of $\Omega_*$, the homological counterpart of algebraic cobordism which is defined for all quasi-projective schemes. Namely we can set $[Z]_{\Omega_*}:=\pi_Z^*(1)$,  where 1 is viewed as the multiplicative unit of the coefficient ring $\LL$. Then, as an element of $\Omega^*(X)$, the fundamental class of $[Z]_{\Omega^*}$ is given by $i_{Z*}([Z]_{\Omega_*})$, which as a cobordism cycle can be rewritten as $[Z\stackrel{i_Z}\irarrow X]$. It is  worth noting that, since $id_{X*}=id_{\Omega^*(X)}$, for $Z=X$ one recovers the original definition for smooth schemes $1_X:=[X\stackrel{id_X}\longrightarrow X]$.

Let us finish this section by warning the reader that we will follow the common practice of writing $[X]_\Omega$ instead of the more precise notation $[X\stackrel{\pi_X}\longrightarrow \Spec(\bfk)]$ when dealing with the elements of the coefficient ring $\Omega^*(\Spec(\bfk))$. More generally, the subscript $\Omega$ will from now on be omitted and, unless stated otherwise, all classes are to be thought of as cobordism classes. Finally, we will consider the Lazard ring $\LL$ as a graded ring  in view of the isomorphism with $\Omega^*(\Spec(\bfk))$. \emph{For the rest of the paper, we will work with algebraic cobordism $\Omega^*$ and $F(u,v)\in \LL[[u,v]]$ will stand for the universal formal group law.}

\section{Segre classes and relative Segre classes}\label{secSegre}
In this section we first introduce Segre classes for algebraic cobordism and compute their generating function (Theorem \ref{thmSegre}). Then we use such description to define relative Segre classes, which we later describe in Theorem \ref{relSegrePush} as pushforwards of Chern classes along a projective bundle. This will be the main ingredient for the computation of the Damon--Kempf--Laksov classes in Section \ref{secKLdet}.
\subsection{Definition of $w(E;u)$}\label{secw}
In order to describe the generating function of Segre classes, we introduce $w_{-s}(E)$, whose definition is based on the following elementary observation.
\begin{lem}\label{lemP}
There exists a unique power series $P(z,x)\in \Laz[[z,x]]$ of degree $0$ and constant term~$1$ satisfying
\[
F(z,\chi(x)) = (z-x) P(z,x).
\]
\end{lem}
\begin{proof}
Let us write $F(z,\chi(x))=\sum_{j=0}^{\infty} Q_j(z,x)$ where each $Q_j(z,x)$ is a homogeneous polynomial of total degree $j$ in $z$ and $x$. Each $Q_j(z,x)$ becomes $0$ if one sets $z=x$, thus it is divisible by $(z-x)$. Therefore the claim holds.   
\end{proof}
\begin{defn}\label{defwtilde1}
Let $\bfx=\{x_1,\dots, x_n\}$ be a set of formal variables. For each integer $s \geq 0$, we define $w_{-s}(\bfx) \in \LL[[\bfx]]$ 
by
\[
\prod_{q=1}^n P(z, x_{q}) =\sum_{s= 0}^{\infty} {w}_{-s}(\bfx) z^{s}
\]
and let $w(\bfx;u):=\sum_{s\geq 0}^{\infty} {w}_{-s}(\bfx)u^{-s}$. If $x_1,\dots, x_e$ are interpreted as the Chern roots of a vector bundle $V$, then we can define $w(V;u):=w(\bfx;u)$ and $w_{-s}(V):={w}_{-s}(\bfx)$. 
\end{defn}
Since $w_0(\bfx)$ has constant term $1$, it is invertible in $\LL[[\bfx]]$. Moreover, an easy computation yields
\begin{equation}\label{defwtilde2}
c_n(L\otimes V^{\vee}) = \prod_{q=1}^nF(z, \chi(x_{q}))  = \sum_{p=0}^n(-1)^pc_p(V) z^{n-p} w(V;z^{-1}).
\end{equation}
\subsection{Segre classes}
\begin{defn}\label{defSegre}
Let $V$ be a vector bundle of rank $n$ over $X$. For each $k \in \ZZ$, consider the dual projective bundle $\pi_m: \PP^*(V\oplus O_X^{\oplus m}) \to X$ for some $m\geq \max\{0, -k-n+1\}$ where $O_X$ is the trivial line bundle over $X$. We then define the degree $k$ \emph{Segre class} $\scS_k(V)$ of $V$ by 
\[
\scS^{(m)}_k(V) = \pi_{m*}(\tau^{k+n+m-1}), 
\]
where $\tau$ is the first Chern class of the tautological quotient line bundle $\calQ$ of $\PP^*(V\oplus O_X^{\oplus m})$. 
\end{defn}
\begin{rem}
 It is easy to see that the definition of $\scS_k(V)$ is actually independent of $m$. In fact, for $m'>m$ one has a canonical inclusion $\iota_m^{m'}:\PP^*(V\oplus O_X^{\oplus m})\irarrow \PP^*(V\oplus O_X^{\oplus m'})$, whose associated pushforward map $(\iota_m^{m'})_*$ is just multiplication by $\tau^{m'-m}$. Then, since $\pi_m=\pi_{m'}\circ \iota_m^{m'}$,  one has
$$\scS_k^{(m)}(V)=\pi_{m*}(\tau^{k+n+m-1})=\pi_{m'*}\big((\iota_m^{m'})_*(\tau^{k+n+m-1})\big)=\pi_{m'*}(\tau^{k+n+m'-1})=\scS_k^{(m')(V)}$$
and it is therefore possible to remove the superscript $(m)$ from the notation. 

\end{rem}
\begin{rem}
If $V$ is a line bundle and $m=0$, we have $\calQ=V$ and $\pi=\id_X$, {\it i.e.} $\scS_k(V)=c_1(V)^k$ for all $k\geq 0$.
\end{rem}
\begin{thm}\label{thmSegre}
Let $V$ be a vector bundle of rank $n$ over $X\in \SM$ and $\scS(V;u) = \sum_{k\in \ZZ} \scS_k(V)u^k$. Then we have
\[
\scS(V;u) =\frac{\scP(u)}{c(V;-u)w(V;u)},
\]
where we set
\[
\scP(u):=\sum_{i=0}^{\infty} [\PP^i]u^{-i}.
\]
where $[\PP^i]\in \Laz^{-i}$ is the class of the projective space $\PP^i$.
\end{thm}
\begin{proof}
We will prove
\[
\scP(u)=c(V;-u)w(V;u)\scS(V;u).
\]
Let us begin by proving the equalities in degree $-m$, for $m\in\NN$. For this we consider the vector bundle $V\oplus O_X^{m+1}\rightarrow X$ and its projectivization $\pi_{m+1}: \PP^*(V\oplus O_X^{m+1}) \to X$. Let $\calQ \to \PP^*(V\oplus O_X^{m+1})$ denote its universal quotient line bundle and $\tau$ its first Chern class.

 The composition of the bundle maps $\pi^*_{m+1}V\irarrow \pi^*_{m+1}(V\oplus O_X^{m+1}) \surj \calQ$ yields a section $s_{m+1}:\PP^*(V\oplus O_X^{m+1})\rightarrow \pi_{m+1}^*V^\vee\otimes \calQ$. Since $m+1\geq 0$, we can identify the zero locus of this section: $Z(s_{m+1})\simeq \PP^*(O_X^{m+1})\simeq(\PP^m)^*\times_{\Spec(\bfk)} X$. Moreover, as its codimension in $\PP^*(V\oplus O_X^{m+1})$ is $n$, its fundamental class is given by the top Chern class of $\pi_{m+1}^*V^\vee\otimes \calQ$. Hence, together with (\ref{defwtilde2}),  we obtain 
\begin{eqnarray*}
[Z(s_{m+1}) \to \PP^*(V\oplus O_X^{m+1})]&=&
c_n(\pi_{m+1}^*V^\vee\otimes \calQ)
=\sum_{k=0}^{\infty}\sum_{i=0}^n (-1)^{n-i}c_{n-i}(V)w_{i-k}(V)\tau^k.
\end{eqnarray*}
Now we push-forward this equality to $\Omega^*(X)$ and get
\begin{eqnarray*}[\PP^m]\cdot 1_X
&=&\sum_{k=0}^{\infty}\sum_{i=0}^n (-1)^{n-i}c_{n-i}(V)w_{i-k}(V)\pi_{m*}(\tau^k) \\
&=&\sum_{k=0}^{\infty}\sum_{i=0}^n (-1)^{n-i}c_{n-i}(V)w_{i-k}(V)\scS_{k-n-m}(V),
\end{eqnarray*}
which is precisely the desired equality.

Let us now focus on the equalities in degree $m$, with $m$ strictly positive. To do this we consider the projective bundle $\pi: \PP^*(V) \to X$ and the following short exact sequence of vector bundles:
\[
0\rightarrow \calQ^\vee\rightarrow \pi^* V^\vee\rightarrow H^\vee\rightarrow 0
\]
over $\PP^*(V)$. By twisting it by $\calQ$, we see that the first term is trivial and as a consequence we get  $c_n(\pi^*V^\vee\otimes \calQ)=0$. As in the previous part, we expand the left hand side by means of the Chern polynomials and of the power series $w(V,u)$ . Hence we obtain
\begin{eqnarray*}
c_n(\pi^*V^\vee\otimes \calQ)&=&\sum_{k=0}^{\infty}\sum_{i=0}^n (-1)^{n-i}c_{n-i}(V)w_{i-k}(V)\tau^k,
\end{eqnarray*}
where we set $\tau:=c_1(\calQ)$. It now suffices to multiply both sides by $\tau^m$ and push them forward to $\Omega^*(X)$ to obtain the desired equality in degree $m$.
\end{proof}
\begin{exm}
For connective $K$-theory $\CK^*$, we have $\scP(u)=\frac{1}{1-\beta u^{-1}}$ and $w(V;u)=\frac{1}{c(V;-\beta)}$. Thus  Theorem \ref{thmSegre} gives 
\[
\scS(V;u) =\frac{1}{1-\beta u^{-1}} \frac{c(V;-\beta)}{c(V;-u)},
\]
which was obtained in \cite{HIMN}. Note that the sign convention for $\beta$ is opposite from the one in \cite{HIMN}.
\end{exm}
\subsection{Relative Segre classes}
Let $0\to V' \to V \to V'' \to 0$ be a short exact sequence of vector bundles. From Definition \ref{defwtilde1}, we can observe that
\[
w(V;u) = w(V'; u) w(V'';u).
\]
This allows us to define the following.
\begin{defn}\label{defRelSeg}
Let $V$ and $W$ be arbitrary vector bundles over $X$. For each $s\geq 0$, we define the class $w_{-s}(V-W)$ in $\Omega^{-s}(X)$ by
\begin{equation*} 
w(V-W;u) :=\sum_{s = 0}^{\infty} w_{-s}(V-W) u^{-s} := \frac{w(V;u)}{w(W;u)}.
\end{equation*}
For each $k\in \ZZ$, we define the relative Segre class $\scS_k(V-W)$ in $\Omega^k(X)$ by
\begin{equation}\label{relSeg}
\scS(V-W;u) := \sum_{k\in \ZZ} \scS_k(V-W)u^k:= \frac{\scP(u)}{c(V-W;-u)w(V-W;u)},
\end{equation}
or equivalently,
\begin{equation*}
\scS_k(V-W) := \sum_{q=0}^{\rk(W)}\sum_{j = 0}^{\infty} (-1)^qc_q(W)w_{-j}(W)\scS_{k-q+j}(V) .
\end{equation*}
Both the classes $w_{-s}(V-W)$ and $\scS_k(V-W)$ are well-defined if $[V-W]$ is viewed as an element of the Grothendieck group of vector bundles over $X$.
\end{defn}

The following description generalises Proposition 2.11 of \cite{HIMN} to algebraic cobordism.
\begin{thm}\label{relSegrePush}
Let $V$ and $W$ be vector bundles over $X$ of rank $n$ and $m$ respectively. Let $\pi: \PP^*(V) \to X$ be the dual projective bundle, $\calQ$ its tautological quotient line bundle, and $\tau:=c_1(\calQ)$. We have
\[
\pi_*(\tau^s c_m(\calQ \otimes W^{\vee}))  = \scS_{m-n+1+s}(V-W).
\]
\end{thm}
\begin{proof}
In view of (\ref{defwtilde2}) one gets
\begin{eqnarray*}
\tau^s c_m(\calQ\otimes W^{\vee}) 
&=& \sum_{q=0}^m\sum_{j = 0}^{\infty}(-1)^qc_q(W) w_{-j}(W)\tau^{j+m-q+s}.
\end{eqnarray*}
Thus, by the definition of $\scS_k(V)$, we have
\begin{eqnarray*}
\pi_*(\tau^s c_m(\calQ \otimes W^{\vee})) 
&=&\sum_{q=0}^m \sum_{j=0}^{\infty}  (-1)^{q}c_{q}(W) w_{-j}(W) \scS_{m-n+1+s-q+j}(E),
\end{eqnarray*}
the right hand side of which is $\scS_{m-n+1+s}(W-V)$ by (\ref{relSeg}).
\end{proof}
\begin{rem}\label{remLINE}
If $V$ is a line bundle, then $\pi = \id_X$ as mentioned above. In this case, we have $\tau^s c_m(\calQ \otimes W^{\vee})  = \scS_{m-n+1+s}(V-W)$.
\end{rem}
\section{Grassmannian degeneracy loci and Damon--Kempf--Laksov classes}\label{secKLdet}
Let $E$ be a vector bundle of rank $e$ over a smooth quasi-projective variety $X$. Let $\Gr_d(E) \to X$ be the Grassmann bundle over $X$ consisting of pairs $(x,U_x)$ where $x \in X$ and $U_x$ is a $d$-dimensional subspace of $E_x$, the fibre of $E$ at $x$. Let $U$ be the tautological bundle of $\Gr_d(E)$. Fix a complete flag $0=F^e\subset \cdots \subset F^1\subset F^0= E$ where $\rk\ F^k = e-k$. We set $F_k:=E/F^k$. In the rest of the paper we will suppress from the notation the pullback of bundles.

 A partition $\lambda$ with at most $d$ parts is a weakly decreasing sequence $(\lambda_1,\dots,\lambda_d)$ of nonnegative integers. The length of $\lambda \in \calP_d$ is the number of nonzero parts, where $\calP_d$ is the set of all partitions with at most $d$ parts.  Let $\calP_d(e)$ be the set of all partitions $\lambda$ in $\calP_d$ such that $\lambda_1\leq e-d$. For each $\lambda \in \calP_d(e)$ of length $r$, we define the degeneracy locus $X_{\lambda}$ in $\Gr_d(E)$ by
\[
X_{\lambda}:=\left\{(x,U_x) \in \Gr_d(E) \ |\ \dim (F^{\lambda_i-i+d}_x \cap U_x) \geq i, i=1,\dots, r\right\}.
\]

Consider the $r$-step flag bundle $\Fl_r(U)$ of $U$ over $\Gr_d(U)$, whose fiber at $(x,U_x)$ is a flag of subspaces $(D_1)_x\subset \cdots \subset (D_r)_x \subset U_x$ with $\dim (D_i)_x=i$. Let $D_1\subset \cdots \subset D_r$ be the tautological bundles of $\Fl_r(U)$ and set $D_0=0$. The flag bundle $\Fl_r(U)$ can be realised as the following tower of projective bundles
\begin{eqnarray}
&&\pi: \Fl_r(U)=\PP(U/D_{r-1}) \stackrel{\pi_r}{\longrightarrow} 
\PP(U/D_{r-2}) \stackrel{\pi_{r-1}}{\longrightarrow} \cdots \ \ \ \ \ \ \ \ \ \nonumber\\\label{tower}
&&\ \ \ \ \ \ \ \ \ \ \ \ \ \ \ \ \ \ \ \ \ \ \cdots\stackrel{\pi_3}{\longrightarrow} \PP(U/D_1) \stackrel{\pi_2}{\longrightarrow} \PP(U)  \stackrel{\pi_1}{\longrightarrow} \Gr_d(E).
\end{eqnarray}
We regard $D_i/D_{i-1}$ as the tautological line bundle of $\PP(U/D_{i-1})$. Denote $\tau_i:=c_1((D_i/D_{i-1})^{\vee})$.
\begin{defn}\label{defKL}
For each $\lambda \in \calP_d(n)$ of length $r$, define a subvariety $Y_{\lambda} \subset \Fl_r(U)$ by
\[
Y_{\lambda} := \left\{ \big(x,U_x, (D_{\bullet})_x\big) \in  \Fl_r(U) \ |\ (D_i)_x \subset F^{\lambda_i+d-i}_x, i=1,\dots,r\right\}.
\]
The cobordism class $[Y_\lambda \to \Gr_d(E)]$ of $Y_{\lambda}$ in $\Gr_d(E)$ is thus defined as the pushforward of the fundamental class of $Y_{\lambda}$ in $\Omega^*(\Fl_r(U))$ along $\pi$, \textit{i.e.}
\[
[Y_\lambda \to \Gr_d(E)] :=\pi_*[Y_\lambda \to \Fl_r(U)].
\]
\end{defn}

\begin{rem}
It is well-known that $Y_\lambda$ is smooth and birational to $X_{\lambda}$ along $\pi$. Since $X_\lambda$ has at worst rational singularities it follows that the specialisation of the class $[Y_\lambda \to \Gr_d(E)]$ to $\CK^*(\Gr_d(E))$ coincides with the fundamental class $[X_{\lambda}]_{\CK}$ of $X_{\lambda}$ (cf. \cite{HIMN}).
\end{rem}

We can express the class of $Y_{\lambda}$ in $\Omega^*(\Fl_r(U))$ as follows.
\begin{prop}\label{propKL}
In $\Omega^*(\Fl_r(U))$, we have
\begin{equation}\label{Y_r prod}
[Y_\lambda \to \Fl_r(U)]=\prod_{i=1}^r c_{\lambda_j+d-j}\Big((D_j/D_{j-1})^{\vee}\otimes F_{\lambda_j+d-j}\Big).
\end{equation}
\end{prop}
This is the exact analogue for $\Omega^*$ of \cite[Lemma 3.3]{HIMN} and the proof, which we omit, can be easily obtained by making use of the following lemma.
\begin{lem}[(\protect{\cite[Lemma 6.6.7]{LevineMorel},\cite[Example 14.1.1]{FultonIntersection}})]\label{gaussbonnet}
Let $V$ be a vector bundle of rank $n$ over $X$ and $s$ a section of $V$. Let $Z$ be the zero scheme of $s$. If $X$ is Cohen-Macaulay and the codimension of $Z$ in $X$ is $n$, then $s$ is regular and 
\[
c_m(V)=[Z\to X] \in \Omega^n(X).
\]
\end{lem}
We will compute the class $[Y_\lambda \to \Gr_d(E)]$ by pushing forward the product of Chern classes (\ref{Y_r prod}) through the tower of projective bundles (\ref{tower}). First we need some algebraic preparations, following \cite{HIMN}. 
Set $R=\Omega^*(\Gr_d(E))$, viewed as a graded algebra over $\LL$. Let $t_1,\ldots,t_{r}$ be indeterminates of degree $1$. We use the multi-index notation $t^\bfs:=t_1^{s_1}\cdots t_{r}^{s_{r}}$ for $\bfs=(s_1,\dots,s_{r})\in \ZZ^{r}$. A formal Laurent series $f(t_1,\ldots,t_{r})=\sum_{\bfs\in\ZZ^{r}}a_{\bfs}t^{\bfs}$ is {\em homogeneous of degree} $m\in \ZZ$ if $a_{\bfs}$ is zero unless $a_{\bfs}\in R_{m-|\bfs|}$ with $|\bfs|=\sum_{i=1}^{r} s_i$. Let $\supp\, f = \{\bfs \in \ZZ^r \ |\ a_{\bfs}\not=0\}$.
For each $m \in \ZZ$, define $\calL^{R}_m$ to be the space of all formal Laurent series of homogeneous degree $m$ such that there exists $\sfn\in \ZZ^r$ such that $\sfn + \supp\, f$ is contained in the cone in $\ZZ^r$ defined by $s_1\geq0, \; s_1+s_2\geq 0, \;\cdots, \; s_1+\cdots + s_{r} \geq 0$. Then $\calL^{R}:=\bigoplus_{m\in \ZZ} \calL^{R}_m$ is a graded ring over $R$ with the obvious product. 
For each $i=1,\dots, r$, let $\calL^{R,i}$ be the $R$-subring of $\calL^R$ consisting of series that do not contain any negative powers of $ t_1,\dots, t_{i-1}$.  In particular, $\calL^{R,1}=\calL^{R}$. 
A series $f(t_1,\ldots,t_{r})$ is a {\em power series} if it doesn't contain any negative powers of $t_1,\dots,t_r$. Let $R[[t_1,\ldots,t_r]]_{m}$ denote the set of all power series in $t_1,\dots, t_r$ of degree $m\in \ZZ$. We set $R[[t_1,\ldots,t_r]]_{\gr}:=\bigoplus_{m\in \ZZ}R[[t_1,\ldots,t_r]]_{m}$.
\begin{defn}
For each $j=1,\dots, r$, define a graded $R[[ t_1,\dots, t_{j-1}]]_{\gr}$-module homomorphism
\[
\phi_j: \calL^{R,j} \to \Omega^*\big(\PP(U/D_{j-2})\big)
\]
by setting 
\[
\phi_j( t_1^{s_1}\cdots  t_{r}^{s_d})= \tau_1^{s_1}\cdots  \tau_{j-1}^{s_{j-1}}\scS_{s_j}(j) \cdots \scS_{s_r}(r)
\]
where $\scS_m(i) := \scS_m((U- F_{\lambda_i-i+d})^{\vee})$ for $m\in \ZZ$ and $i=1,\dots,r$. It is known that $\Omega^*(\PP(U/D_{j-2}))$ is bounded above, \textit{i.e.}, $\Omega^m(\PP(U/D_{j-2})) = 0$ for all $m > \dim \PP(U/D_{j-2})$. Therefore $\scS_m(i)$ is zero for all sufficiently large $m$. This ensures that the above map is well-defined.
\end{defn}
We have the following pushforward formula for each stage of the tower in terms of $\phi_j$.
\begin{lem}\label{pphi}
Let $\alpha_j:=c_{\lambda_j+d-j}\big((D_j/D_{j-1})^{\vee}\otimes F_{\lambda_j+d-j}\big)$ for $j=1,\dots, r$. For each non-negative integer $s$, we have
\[
\pi_{j*}\big(\tau_j^s\alpha_j\big)=\phi_j\left(t_j^{\lambda_j+s}\prod_{i=1}^{j-1}(1 - t_i/t_j) P(t_j, t_i) \right).
\]
\end{lem}
\begin{proof}
 By applying Theorem \ref{relSegrePush} to $\pi_j: \PP(U/D_{j-1}) \to \PP(U/D_{j-2})$ the left hand side can be evaluated as 
\[
\pi_{j*}(\tau_j^s\alpha_j)=\scS_{\lambda_j+s}\big((U/D_{j-1} - F_{\lambda_j+d-j})^{\vee}\big)=\scS_{\lambda_j+s}\big((U-F_{\lambda_j-j+d})^{\vee}-D_{j-1}^{\vee}\big).
\]
From the definition of the relative Segre class (\ref{relSeg}), we obtain
\begin{eqnarray*} 
\pi_{j*}\big(\tau_j^s\alpha_j\big)&=& \sum_{p=0}^{j-1}\sum_{q=0}^{\infty}(-1)^pc_p(D_{j-1}^{\vee}) w_{-q}(D_{j-1}^{\vee})\scS_{\lambda_j+s-p+q}(j).
\end{eqnarray*}
Thus by using $\phi_j$, we have
\begin{eqnarray*}
\pi_{j*}\big(\tau_j^s\alpha_j\big)
&=&\phi_j\left(\sum_{p=0}^{j-1}\sum_{q=0}^{\infty}(-1)^pe_p(t_1,\dots,t_{j-1}) w_{-q}(t_1,\dots,t_{j-1})  t_j^{\lambda_j+s-p+q}\right)\\
&=&\phi_j\left( t_j^{\lambda_j+s}\left(\sum_{p=0}^{j-1}(-1)^pe_p(t_1,\dots,t_{j-1})t_j^{-p} \right)\left(\sum_{q=0}^{\infty}w_{-q}(t_1,\dots,t_{j-1})t_j^q\right) \right).
\end{eqnarray*}
The claim follows from the definitions of $w_{-q}$ and of the elementary symmetric polynomials $e_p$ in terms of their generating functions.
\end{proof}
Now we obtain our main application. Let
\[
\Delta_{\bfm}\Big(\scS(1), \dots, \scS(r)\Big) := \det\Big(\scS_{m_i+j-i}(i)\Big)_{1\leq i,j\leq r}
\]
for each $\bfm=(m_1,\dots,m_r)\in \NN^r$. Let $a_{\bfs}\in \LL$ be the coefficients of the power series
\begin{equation*}
\prod_{1\leq i<j\leq r}P(t_j,t_i)=\sum_{\bfs=(s_1,\dots, s_r) \in\NN^r}a_{\bfs}\cdot t_1^{s_1}\cdots t_r^{s_r}
\end{equation*}
as an element of $\calL^{\LL}$. 
\begin{thm}\label{detthm}
For a partition $\lambda \in \calP_d(n)$ of length $r$, the class $[Y_\lambda \to \Gr_d(E)]$ is given by
\[
[Y_\lambda \to \Gr_d(E)]=\sum_{\bfs=(s_1,\dots, s_r) \in\NN^r}a_{\bfs} \Delta_{\lambda+\bfs}\Big(\scS(1), \dots, \scS (r)\Big).
\]
\end{thm}
\begin{proof}
By Definition \ref{defKL} and Proposition \ref{propKL}, we have
\[
[Y_\lambda \to \Gr_d(E)] = \pi_{1*} \circ \cdots \circ \pi_{r*}\left(\prod_{j=1}^r \alpha_i\right)
\]
 and a repeated application of Lemma \ref{pphi} (\textit{cf.} \cite[Section 4.4]{HIMN}) yields
\[
[Y_\lambda \to \Gr_d(E)] = \phi_1\left(t_1^{\lambda_1}\cdots t_r^{\lambda_r} \prod_{1\leq i<j\leq r} (1-t_i/t_j)\prod_{1\leq i<j\leq r} P(t_j,t_i)  \right).
\]
Since $\phi_1$ is linear, we have
\[
[Y_\lambda \to \Gr_d(E)] = \sum_{\bfs=(s_1,\dots, s_r) \in\NN^r}a_{\bfs}
\phi_1\left(t_1^{\lambda_1+s_1}\cdots t_r^{\lambda_r+s_r} \prod_{1\leq i<j\leq r} (1-t_i/t_j)\right).
\]
Vandermode's determinant formula allows us to write
\[
t_1^{\lambda_1+s_1}\cdots t_r^{\lambda_r+s_r} \prod_{1\leq i<j\leq r} (1-t_i/t_j) = \det\left(t_i^{\lambda_i+s_i+j-i}\right),
\]
thus by applying $\phi_1$ we obtain
\[
[Y_\lambda \to \Gr_d(E)]=\sum_{\bfs=(s_1,\dots, s_r) \in\NN^r}a_{\bfs} \det\Big(\scS_{\lambda_i+s_i+j-i}(i)\Big)_{1\leq i,j\leq r}.
\]
This completes the proof.
\end{proof}
In connective $K$-theory $[Y_\lambda \to \Gr_d(E)]_{CK}$ coincides with the fundamental class of the degeneracy locus $X_{\lambda}$ and thus Theorem \ref{detthm} implies the following determinantal formula describing $[X_{\lambda}]_{CK}$, which is different from the one obtained in \cite{HIMN}.
\begin{cor}\label{exmCK}
For a partition $\lambda\in \calP_d(n)$, we have
\begin{eqnarray*}
[X_{\lambda}]_{\CK}
&=&\det\left(\sum_{s\geq 0}  \binom{i-r}{s}(-\beta)^s \scS_{\lambda_i+j-i+s}([i])\right)_{1\leq i,j\leq r}.
\end{eqnarray*}
\end{cor}
\begin{proof}
For connective $K$-theory one has $P_{CK}(x,y)=\frac{1}{1-\beta y}$, so in this case the formula follows from the identity
\begin{equation*}
t_1^{\lambda_1}\cdots t_r^{\lambda_r} \prod_{1\leq i<j\leq r} (1-t_i/t_j)\prod_{1\leq i<j\leq r} P(t_j,t_i) 
=\det\left(\left(\frac{1}{1-\beta t_i}\right)^{r-i}  t_i^{\lambda_i+j-i}\right)_{1\leq i,j\leq r}.\qedhere
\end{equation*}
\end{proof}

\textit{Acknowlegdements:} Both authors would like to thank Takeshi Ikeda for useful discussions and Marc Levine for his valuable comments which greatly improved the readability.
\bibliographystyle{ieeetr}
\bibliography{references}{}

\begin{small}
{\scshape
\noindent Thomas Hudson, Fachgruppe Mathematik
und Informatik, Bergische Universit\"{a}t Wuppertal, Gaußstrasse 20, 42119 Wuppertal, Germany
}
\end{small}

{\textit{email address}: \tt{hudson@math.uni-wuppertal.de}}

\

\begin{small}
{\scshape
\noindent Tomoo Matsumura, Department of Applied Mathematics, Okayama University of Science, Okayama 700-0005, Japan
}
\end{small}

{\textit{email address}: \tt{matsumur@xmath.ous.ac.jp}}

\end{document}